\documentclass{amsart}
\usepackage{amssymb}

\usepackage{lmodern}
\usepackage[T1]{fontenc}

\usepackage[headings]{fullpage}
\usepackage{paralist}
\usepackage[usenames]{color}
\usepackage{xspace}
\usepackage[normalem]{ulem}
\usepackage[colorlinks,pagebackref=true,pdftex]{hyperref}
\usepackage[alphabetic,abbrev,backrefs,lite]{amsrefs}

\makeatletter

\@addtoreset{equation}{section}
\def\theequation{\thesection.\@arabic \c@equation}

\def\theenumi{\@roman\c@enumi}

\def\@citecolor{blue}
\def\@urlcolor{blue}
\def\@linkcolor{blue}

\newcounter{claim}
   {\refstepcounter{claim}\def\@currentlabel{\thethm\alph{claim}}%
   \underline{Claim \@currentlabel}:\itshape}{}
\makeatother

\theoremstyle{plain}
\newtheorem{thm}[equation]{Theorem}
\newtheorem{lemma}[equation]{Lemma}

\newtheorem{propn}[equation]{Proposition}

\theoremstyle{definition}
\newtheorem{defn}[equation]{Definition}
\newtheorem{question}[equation]{Question}

\theoremstyle{remark}
\newtheorem{obs}[equation]{Observation}

\newtheorem{remark}[equation]{Remark}
\newenvironment{remarkbox}[1][]{%
    \begin{remark}[#1] \pushQED{\qed}}{\popQED \end{remark}}

\newtheorem{example}[equation]{Example}
\newenvironment{examplebox}[1][]{%
    \begin{example}[#1] \pushQED{\qed}}{\popQED \end{example}}

\newtheorem{definition}[equation]{Definition}
\newenvironment{definitionbox}[1][]{%
    \begin{definition}[#1]\pushQED{\qed}}{\popQED \end{definition}}

\newtheorem{notation}[equation]{Notation}

\newtheorem{discussion}[equation]{Discussion}
\newenvironment{discussionbox}[1][]{%
    \begin{discussion}[#1]\pushQED{\qed}}{\popQED \end{discussion}}

\newtheorem{construction}[equation]{Construction}

\newcommand{\naturals}{\mathbb{N}}
\newcommand{\ints}{\mathbb{Z}}
\newcommand{\complex}{\mathbb{C}}

\DeclareMathOperator{\Tor}{Tor}
\DeclareMathOperator{\In}{in}

\DeclareMathOperator{\hilbertPoset}{\mathcal H}
\DeclareMathOperator{\idealPoset}{\mathcal I}
\DeclareMathOperator{\monomials}{Mon}

\newcommand{\define}[1]{\emph{#1}}
\newcommand{\minus}{\ensuremath{\smallsetminus}}
\newcommand{\lex}{\ensuremath{\text{\textsc{lex}}}\xspace}
\def\hssymb{\mathfrak z}

\title{Poset Embeddings of Hilbert Functions}
\author[G.~Caviglia]{Giulio Caviglia}
\address{Department of Mathematics, Purdue University, West Lafayette IN
47907}

\email{gcavigli@math.purdue.edu}

\author[M.~Kummini]{Manoj Kummini}
\address{Department of Mathematics, Purdue University, West Lafayette IN
47907}
\curraddr{Chennai Mathematical Institute, Siruseri, Tamilnadu 603103.
India.}
\email{mkummini@cmi.ac.in}
\makeindex
\begin{document}

\begin{abstract}
For a standard graded algebra $R$,
 we consider embeddings of the poset
of Hilbert functions of $R$-ideals into the poset of $R$-ideals,
as a way of classification of Hilbert functions. There are examples of
rings for which such embeddings do not exist. We describe how the embedding
can be lifted to certain ring extensions, which is then used in the case of
polarization and distraction. A version of a theorem of
Clements--Lindstr\"om is proved. We exhibit a condition on the embedding
that ensures that the classification of Hilbert functions is obtained with
images of lexicographic segment ideals.
\end{abstract}

\maketitle

\section{Introduction}
Let $R$ be a \define{standard graded algebra} over a field $\Bbbk$,
i.e., $R \simeq \bigoplus_{d \geq 0}R_d$ as $\Bbbk$-vector-spaces
with $R_0 = \Bbbk$, $R = \Bbbk[R_1]$ and $\dim_\Bbbk R_1 < \infty$. When
$R$ is a polynomial ring, a theorem of F.~Macaulay (see,
e.g.,~\cite{BrHe:CM}*{Section~4.2}) 
provides a classification of
the Hilbert functions of homogeneous $R$-ideals; more precisely, 
a function $H : \naturals \to \naturals$ is the Hilbert function of some
homogeneous ideal if and only if it is the Hilbert function of an ideal
generated by \lex-segments.
(\lex denotes the graded lexicographic monomial order on a polynomial ring.)
\lex-segment ideals in polynomial
rings have been extensively studied. It is known that such ideals have
several extremal properties;
see~\cites{BigaUpperBds93, HuleMaxBettiNos93, PardueDefClass96,
SbarraBdsLocalCoh01, ConcaExtremalGinLex04, CHHRigidReslns04}.

Let $A = \Bbbk[x_1, \ldots, x_n]$ be a polynomial ring, $\mathfrak a =
(x_1^{e_1}, \ldots, x_n^{e_n})$ and $R = A/\mathfrak a$.
J.~Kruskal~\cite{KruskalNoOfSimpl63} and G.~Katona~\cite{KatonaFiniteSets68}
(for the case $e_1 = \cdots = e_n = 2$) and G.~Clements and
B.~Lindstr\"om~\cite{ClemLindMacaulayThm69} (more generally, for all $2
\leq e_1 \leq \cdots \leq e_n \leq \infty$) proved that every homogeneous
$R$-ideal has the same Hilbert function as the image (in $R$)
of a $\lex$-segment $A$-ideal. The following conjecture of D.~Eisenbud,
M.~Green and J.~Harris furthered the interest in studying images of
\lex-segment ideals in quotient rings: let $A$ be the polynomial ring, as
above, and let $I$ be a homogeneous $A$-ideal containing an
$A$-regular sequence
of homogeneous polynomials $f_1, \ldots, f_n$ of degrees $e_1 \leq \cdots
\leq e_n$; then there exists a \lex-segment ideal $L$ such that the
Hilbert functions of $L + (x_1^{e_1}, \ldots, x_n^{e_n})$ and $I$ are
identical. (See~\cite{EGHcbconj96}*{Conjecture~CB12} for the original
formulation, and \cite{FrRiLPP07} for a more recent survey.) In a similar
vein, V.~Gasharov, N.~Horwitz, J.~Mermin, S.~Murai and I.~Peeva
studied algebras $R = A/\mathfrak a$ (where $\mathfrak a$ is graded
$A$-ideal) for
which every possible Hilbert function is attained by the images (in $R$) of
\lex-segment $A$-ideals: quotients by compressed-monomial-plus-powers
ideals~\cite{MerminPeevaLexifying}, rational normal
curves~\cite{GHPtoric08}, Veronese rings~\cite{GMPveronese10} and quotients
by coloured square-free monomial ideals~\cite{MeMuColoured10}. In these
papers, such rings are called Macaulay-lex, to emphasize the fact that
every Hilbert function is attained by the image of a \lex-segment ideal,
analogous to the theorem of Macaulay. Mermin, however, showed that most
monomial complete intersections fail to be Macaulay-lex, even after a
reordering of variables~\cite{MerminMonomCI10}*{Theorem~4.4}. 

Motivated by these results, we consider two problems: 
\begin{asparaenum}
\item Is there another approach to classification of Hilbert functions in
quotient rings?
\item What is the significance of \lex-segment ideals?
\end{asparaenum}
To study this, we look at certain embeddings of the poset of Hilbert
functions into the poset of $R$-ideals. When such embeddings exist, they
induce a filtration of the $R_d$ by $\Bbbk$-subspaces; for some $R$,
this results in a degree-wise total order (which we call an
\define{embedding order}) on a standard basis of $R$.  Further, the
embedding order respects multiplication precisely when all 
the Hilbert functions are given by images of \lex-segment ideals.
We now describe this approach in detail.

By $\naturals$ we mean the set of 
non-negative integers.
Let $\idealPoset_R = \{J : J \ \text{is a homogeneous $R$-ideal}\}$,
considered as a poset under inclusion. 
{For $I \in \idealPoset_R$ and $t \in \naturals$, 
we will write $I_t$ for the $\Bbbk$-vector-space of
the homogeneous elements of $I$ of degree $t$. 
(Note that $I \simeq \bigoplus_{t\in \mathbb N} I_t$, as
$\Bbbk$-vector-spaces.) 
The \define{Hilbert function} of $I$ is the function 
$\naturals \to \naturals, t \mapsto \dim_\Bbbk I_t$ 
the \define{Hilbert series} of
$I$ is the formal power series } 
\[
H_I(\hssymb) = \sum_{t \in \naturals} \left(\dim_\Bbbk I_t\right) \hssymb^t \in
\ints[[\hssymb]].
\]
 
For $H \in \ints[[\hssymb]]$, we write $H^t$ for the
coefficient of $\hssymb^t$ in $H$, so $H = \sum_t H^t \hssymb^t$. The
\define{poset of Hilbert series of $R$-ideals} is the set
$\hilbertPoset_R = \{H_{J} : J \in \idealPoset_R\}$ endowed with the
partial order: $H \succcurlyeq H' \in \hilbertPoset_R$ if,
for all $t \in \naturals$, $H^t \geq (H')^t$.
For the sake of convenience, we work with $\hilbertPoset_R$ instead of the 
analogous poset of Hilbert functions of $R$-ideals.
\begin{question} \label{qn:isThereAnEmbedding} 
Is there an (order-preserving) embedding $\epsilon : \hilbertPoset_R
\hookrightarrow \idealPoset_R$, as posets, such that 
${\mathbf H} \circ \epsilon = \mathrm{id}_{\hilbertPoset_R}$, 
where ${\mathbf H}:\idealPoset_R \longrightarrow
\hilbertPoset_R$ is the function $J \mapsto H_{J}$?
\end{question}

We will say that $\hilbertPoset_R$ \define{admits an embedding into}
$\idealPoset_R$ (and often, by abuse of terminology, merely that
$\hilbertPoset_R$ \define{admits an embedding}) if
Question~\ref{qn:isThereAnEmbedding} has an affirmative answer. 
For example, if every Hilbert series in $\hilbertPoset_R$ is attained by
the image of a \lex-segment $A$-ideal, then 
$\hilbertPoset_R$ admits an embedding, as is the case when $R=A$
(Macaulay) or $R = A/(x_1^{e_1}, \ldots, x_n^{e_n})$ (Kruskal, Katona,
Clements--Lindstr\"om) or $R$ is belongs to one of the classes of examples
studied by Gasharov, Horwitz, Mermin, Murai and Peeva, listed earlier.
In Proposition~\ref{propn:embHCompl}, we prove a necessary condition for 
$\hilbertPoset_R$ to admit an embedding, using which we exhibit a few
algebras whose posets of Hilbert series fail to admit embeddings. 
Proposition~\ref{thm:multOrderIsLex} describes when the ideals that appear
as images of a given embedding are \lex-segment $R$-ideals.  In
Section~\ref{sec:extnRings}, we prove that an embedding can be lifted to
certain ring extensions (Theorem~\ref{thm:extnRings}), and use it for a
special case of distractions (Proposition~\ref{thm:distrEmb}) and for
polarization (Theorem~\ref{thm:polzEmb}).
Theorem~\ref{thm:ClemLindEmb} is an analogue, in the situation of
embeddings, of the result of Clements--Lindstr\"om mentioned above. In
Section~\ref{sec:stabilization}, we prove some lemmas about obtaining
stable ideals, which can be read independently of the previous sections,
and are used in Section~\ref{sec:extnRings}.

\subsection*{Notation and terminology}
We use~\cite{eiscommalg} as a general reference.  In this paper, $R$ will
always denote a standard graded algebra over a field $\Bbbk$. The
homogeneous maximal ideal of $R$
is $\mathfrak m = \bigoplus_{d \geq 1}R_d$. By $A$, we mean a (standard
graded) polynomial ring over $\Bbbk$ that has a surjective homogeneous
$\Bbbk$-algebra homomorphism $A \stackrel{\phi}{\longrightarrow} R$ of degree
$0$. We fix this homomorphism.  Let $\mathfrak a =\ker\phi$. We will
further assume that the embedding dimensions of $A$ and $R$ are the same;
equivalently, $\mathfrak a_1 = 0$.  In particular, if $R$ is a polynomial
ring, then $A = R$.

\begin{definitionbox}
(The following definitions depend on the choice of basis of $A_1$.) Fix a
basis $x_1, \ldots, x_n$ of $A_1$.
We write $\monomials(A)$ for the set of monomials in the $x_i$ in $A$. 
By a \define{monomial} of $R$, we mean the image of a monomial of $A$
under $\phi$. 
A \define{standard basis} of $R$ is a set $\mathcal B \subseteq
\monomials(A)$ such that $\{\phi(f) : f \in \mathcal B\}$ 
forms a $\Bbbk$-basis of $R$. Let $\mathcal B$ be a standard basis of $R$;
for $d \in \naturals$, we write $\mathcal B_d$ for the set of monomials of
$\mathcal B$ of degree $d$. For a $\Bbbk$-subspace $V$ of $R_d$, we write
$|V|$ for $\dim_\Bbbk V$. For a subset $V$ of $R$, we write $(V)R$ for the
ideal generated by $V$ in the ring $R$. 
A \define{graded total order} on $R$ is a pair $(\mathcal B, \tau)$ consisting
of a standard basis $\mathcal B$ of $R$ and a total order $\tau$ on
$\mathcal B$ such that $m \prec_\tau m'$ if $\deg m < \deg m'$.  For $d \in
\naturals$, the \define{$\tau$-segment of $R_d$ of dimension} $r$ is
the $\Bbbk$-vector-space generated by the images in $R$ of 
the first $r$ monomials in $\mathcal
B_d$ in the order $\prec_\tau$. 
We say that a graded total order $(\mathcal B,\tau)$ is a
\define{monomial order} if 
\begin{inparaenum}
\item 
for all $f \in \mathcal B$ and for all $g \mid f$, $\frac{f}g
\in \mathcal B$, and,
\item for all $f, f' \in \mathcal B$ and for all 
$g \mid \gcd(f, f')$, $f \prec_\tau f'$ if and only if 
 $\frac{f}g \prec_\tau \frac{f'}g$.
\end{inparaenum}
Suppose that $\mathfrak a$ is a monomial ideal.
Then $\mathcal B = \monomials(A)
\minus \mathfrak a$;
therefore, while referring to any graded total
order, we will drop the reference to the standard basis. 
On the polynomial ring $A$, we will also need to use \define{weight orders}
on the set of monomials, induced by assigning weights inside
$\naturals$ to the $x_i$. Weight orders need not be total orders, in
general.
\end{definitionbox}

\begin{remarkbox}
\label{remarkbox:noRefToBasis}
Note that whether $\hilbertPoset_R$ admits an embedding or not does not
depend on any choice of basis of $R_1$. As we will see in
Proposition~\ref{propn:embIFFfiltr}, the existence of an embedding is
equivalent to a the existence of a certain filtration of $R$ as a
$\Bbbk$-vector-space. However, in Discussion~\ref{discussion:embedOrders}
and what follows, we consider total orders on 
rings defined by monomial ideals and
on certain semigroup rings that correspond to embeddings; in those cases, we
have the basis of $R_1$ given by the images of the variables of $A$ in
mind.
\end{remarkbox}

\begin{remarkbox}
\label{remark:embImMonom}
Suppose that $R$ is defined by a monomial ideal or that it is 
an affine semigroup algebra
all of whose $\Bbbk$-algebra generators are of the same degree (i.e., there is a
injective $\Bbbk$-algebra homomorphism $\xi : R \longrightarrow A'$ for
some polynomial ring $A'$ such that the $\xi(\phi(x_i))$ are monomials in
$A'$, in some fixed basis of $A'_1$, of the same degree). 
If $\hilbertPoset_R$ admits an embedding $\epsilon$, then for all $H \in
\hilbertPoset_R$, we may take $\epsilon(H)$ to be a monomial ideal. To see
this, write $R = A/\mathfrak a$.  We can find a weight order $\omega$ on
$A$ such that $\mathfrak a = \In_\omega(\mathfrak a)$ (the initial ideal of
$\mathfrak a$ with respect to the weight order $\omega$)
and such that for every homogeneous ideal $I$ with
$\mathfrak a \subseteq I$, $\In_\omega(I)$ is generated by $\mathfrak a$
and monomials.  When $R$ is defined by a monomial ideal, this is immediate.
For the details of the latter case, see~\cite{GHPtoric08}*{Theorem~2.5}.
\end{remarkbox}

\section{Generalities}
\label{sec:generalities}

The poset $\idealPoset_R$ of $R$-ideals is a \define{lattice};
if $I, J \in \idealPoset_R$, then their \define{join}, or
least upper bound, is $I \vee J = I + J$ and their \define{meet}, or
greatest lower bound, is $I \wedge J = I \cap J$. (We
use~\cite{StanEC1} as the reference on lattices.) We will see that
if $\hilbertPoset_R$ admits an embedding, then it is a lattice with 
specific meet and join functions. Using this criterion, we show that for
certain standard graded algebras $R$, $\hilbertPoset_R$ admits no
embedding. First, we see in the next lemma that an embedding 
(if it exists) can be done degree-by-degree.

\begin{lemma}
\label{lemma:embHbyDegs}
Suppose that $\hilbertPoset_R$ admits an embedding
$\epsilon$. 
Let $H, \tilde H \in \hilbertPoset_R$. Then for all $d \in \naturals$, if
$H^d \leq \tilde H^d$ then 
$(\epsilon(H))_d \subseteq (\epsilon(\tilde H))_d$. 
\end{lemma}

\begin{proof}
Let $I$ be an $R$-ideal such that $H_I = H$. Define $J$ to be the $R$-ideal 
generated by $I$ and all the forms of degree $d+1$, and let $K$ be the ideal
generated by all the forms of $J$ of degree greater than or equal to $d$.
The fact that $H_J \succcurlyeq H_I$ and $H_J \succcurlyeq H_K$ gives
$\epsilon(H_J) \supseteq \epsilon(H_I)$ and $\epsilon(H_J) \supseteq
\epsilon(H_K),$ while the equalities $H_I^d=H_J^d=H_K^d$ imply that
$(\epsilon(H_I))_d=(\epsilon(H_J))_d= (\epsilon(H_K))_d$. 
Let $I'$ be an $R$-ideal such that $H_{I'} = \tilde H$, 
and define the ideals
$J'$ and $K'$ in a way analogous to above. Since  $H_{K'}
\succcurlyeq  H_K$ we get the desired inequality.
\end{proof}

\begin{remark}
\label{remark:embByDegs}
For a $\Bbbk$-vector-space $V \subseteq R_d$, we write $\epsilon(V)$ for
$(\epsilon(H_{(V)R}))_d$. (This depends only on $|V|$ by
Lemma~\ref{lemma:embHbyDegs}.) Hence for all $I \in \idealPoset_R$,
$\epsilon(H_I) = \bigoplus_{d \in \naturals}\epsilon(I_d)$.
\end{remark}

\begin{defn}
\label{defn:embeddingFlag}
An \define{embedding filtration} of $R$ is a collection of filtrations $\{0 =
V_{d,0} \subsetneq V_{d,1} \subsetneq \cdots \subsetneq V_{d,|R_d|} = R_d :
d \in \naturals\}$ of $R$ into $\Bbbk$-vector-spaces that satisfies, for
all $d \in \naturals$ and for all $0 \leq r \leq |R_d|$,
\begin{asparaenum}
\item $R_1V_{d,r} = V_{d+1,s}$ for some $0 \leq s \leq |R_{d+1}|$, and,
\item for all $W \subseteq R_d$, $|R_1V_{d,|W|}| \leq |R_1W|$.
\end{asparaenum}
\end{defn}

\begin{propn}
\label{propn:embIFFfiltr}
Let $R$ be a standard graded algebra. Then $\hilbertPoset_R$ admits an
embedding into $\idealPoset_R$ if and only if $R$ has an embedding
filtration.
\end{propn}

\begin{proof}
Suppose that $\hilbertPoset_R$ admits an embedding $\epsilon$. 
Let $\mathcal V = 
\{0 = V_{d,0} \subsetneq V_{d,1} \subsetneq \cdots \subsetneq
V_{d,|R_d|} = R_d : d \in \naturals\}$ be a collection of filtrations of
$R$ into $\Bbbk$-vector-spaces. 
For all $d$ and for all $0 \leq r \leq |R_d|$, 
replace $V_{d,r}$ by $\epsilon(V_{d,r})$.
We will show that $\mathcal V$ is an embedding filtration of $R$.

Let $W \subseteq R_d$ be a $\Bbbk$-subspace.  Let $V = V_{d,|W|}$. 
By Lemma~\ref{lemma:embHbyDegs}, 
$V = (\epsilon(H_{(W)R}))_d$, which gives 
$R_1V \subseteq (\epsilon(H_{(W)R}))_{d+1}$.
Note that $|(\epsilon(H_{(W)R}))_{d+1}| = |R_1W|$,
so, $|R_1V| \leq |R_1W|$. 
Now, applying the above calculation with $W = V$, we get
$R_1V \subseteq (\epsilon(H_{(V)R}))_{d+1}$.
By Remark~\ref{remark:embByDegs}, 
$\epsilon(H_{(V)R})_{d} = V$ and
$\epsilon(H_{(V)R})_{d+1} = \epsilon(R_1V)$. Hence
$R_1V \subseteq \epsilon(R_1V)$, so $R_1V = \epsilon(R_1V) = V_{d+1,
|R_1V|}$. 

Conversely, given an embedding filtration $\{0 = V_{d,0} \subsetneq V_{d,1}
\subsetneq \cdots \subsetneq V_{d,|R_d|} = R_d : d \in \naturals\}$, we
define an embedding $\epsilon : \hilbertPoset_R \longrightarrow
\idealPoset_R$ by setting $\epsilon(H) = \bigoplus_{d \in \naturals} V_{d,
H^d}$. It follows, from the definition of embedding filtrations, that
$\epsilon(H)$ is an ideal.
\end{proof}

\begin{remarkbox}
\label{remark:zerothBetti} 
For a graded $R$-module $M$, its \define{graded Betti numbers} are
$\beta_{i,j}^R(M) = \dim_\Bbbk \Tor_i^R(M, \Bbbk)_j$.
Let $\epsilon : \hilbertPoset_R \longrightarrow \idealPoset_R$ be an
embedding. Let $I$ be an $R$-ideal. 
Then there is an
inequality 
$\beta_{1,j}^R(R/I) \leq \beta_{1,j}^R(R/\epsilon(H_{I}))$.
We see this as follows. Note that for any homogeneous
$R$-ideal $J$, $\beta_{1,j}^R(R/J) =  \left(|J_j| - |R_1J_{j-1}|\right)$.
Now let $J = \epsilon(H_{I})$. Let
$\{0 = V_{d,0} \subsetneq V_{d,1} \subsetneq \cdots \subsetneq V_{d,|R_d|}
= R_d : d \in \naturals\}$ be the embedding filtration of $R$ given by
$\epsilon$, as in the proof of Proposition~\ref{propn:embIFFfiltr}. Then,
for all $d \geq 1$, $J_d = V_{d, |I_d|}$. Hence $|R_1J_{d-1}| \leq
|R_1I_{d-1}|$, which implies that $\beta_{1,j}^R(R/I) \leq
\beta_{1,j}^R(R/J)$.
Note, also, that the same argument shows that 
$\beta_{1,j}^R(R/\epsilon(H_{I}))$ depends only on $H_I$
and not on $\epsilon$.
For which $i$ and $j$ is the inequality  
$\beta_{i,j}^R(R/I) \leq \beta_{i,j}^R(R/\epsilon(H_{I}))$ true?
Mermin and Murai~\cite{MeMuColoured10}*{Proposition~3.2}
showed that in general 
$\beta_{i,j}^A(R/I) \not \leq \beta_{i,j}^A(R/\epsilon(H_{I}))$.
\end{remarkbox}
\begin{defn}
Let $H, \tilde H \in \ints[[\hssymb]]$. Define $\max(H, \tilde H)$ and 
$\min(H, \tilde H)$  in $\ints[[\hssymb]]$ by setting, 
for all $t \in \naturals$, 
$(\max(H, \tilde H))^t = \max \{H^t, \tilde H^t\}$
and
$(\min(H, \tilde H))^t = \min \{H^t, \tilde H^t\}$.
\end{defn}

Note that the usual total order on $\ints$ makes 
$\ints[[\hssymb]]$ into a distributive lattice. We now derive a necessary
criterion so that we have an embedding.

\begin{propn}
\label{propn:embHCompl}
If $\hilbertPoset_R$ admits an embedding then it is a sublattice of 
$\ints[[\hssymb]]$.
\end{propn}

\begin{proof}
Let $H, \tilde H \in
\hilbertPoset_R$.
Let $\epsilon : \hilbertPoset_R \longrightarrow \idealPoset_R$ be the
embedding. Let $I, \tilde I \in \idealPoset_R$, $H = H_{I}$ and $\tilde H =
H_{\tilde I}$.  Without loss of generality, we may assume that $I =
\epsilon(H)$ and $\tilde I = \epsilon(\tilde H)$. 
Fix an embedding filtration 
$\{0 = V_{d,0} \subsetneq V_{d,1} \subsetneq \cdots \subsetneq V_{d,|R_d|}
= R_d : d \in \naturals\}$
of $R$.
Write $I = \bigoplus_d V_{d, r_d}$ and $\tilde I = \bigoplus_d V_{d, \tilde {r}_d}$.
Note that $H_{(I+\tilde I)} = H \vee \tilde H$ and $H_{(I \cap \tilde I)} = 
H \wedge \tilde H$.
\end{proof}

\begin{remarkbox}
\label{remark:embInherits}
Suppose that $\hilbertPoset_R$ admits an embedding $\epsilon$. Let 
$I = \epsilon(H)$ for some $H \in \hilbertPoset_R$. 
Then $\idealPoset_{R/I} \simeq \{J
\in \idealPoset_R : I \subseteq J\}$ and 
$\hilbertPoset_{R/I} \simeq \{H_{J} : J \in \idealPoset_R, I \subseteq J\}$. 
In particular, $\epsilon$ induces an embedding of
$\hilbertPoset_{R/I}$ into $\idealPoset_{R/I}$. 
Thus,
if $\mathfrak a$ is a
\lex-segment $A$-ideal then the embedding by \lex-segment $A$-ideals gives
an embedding filtration on $A$; 
the images (in $R$) of \lex-segment $A$-ideals
give an embedding filtration of $R$.
\end{remarkbox}

If $\hilbertPoset_R$ admits an embedding,
then we obtain a complete flag, and hence a basis, of $R_1$. The next
example illustrates this. In the three examples that follow, we have
represented the rings $R$ with a given basis of $R_1$, only for the sake of
concreteness. The assertion that the posets $\hilbertPoset_R$ in those
examples do not admit an embedding is an inherent statement, independent of
the bases.

\begin{examplebox}
Let $R = A/\mathfrak a$ be an Artinian Gorenstein $\Bbbk$-algebra such that
$\mathfrak a$ is generated by quadratic forms in $A$ and $H_R = 1 +
4\hssymb^1 + 4\hssymb^2 + \hssymb^3$. We show that $R$ has an embedding
filtration. In the process of proving that
$R$ is Koszul, A.~Conca, M.~Rossi and G.~Valla observed that there exists a
linear form $l_1 \in R_1$ such that $|R_1l_1| =
2$~\cite{CRVgflags01}*{p.118}, moreover for every other linear form $l$ the inequality $|R_1l| \geq
2$ holds. They further showed (see
\cite{CRVgflags01}*{Lemma~6.14}) 
\begin{inparaenum}
\item that there exists a $2$-dimensional $\Bbbk$-vector-space $V \subseteq R_1$ such that
$l_1 \in V$, and $|R_1V| = 3$, and,
\item for every $2$-dimensional $\Bbbk$-vector-spaces $W\subseteq R_1$ the inequality
$|R_1W|\geq 3$ holds.
\end{inparaenum}
Let $l_2 \in V$ be such that
$l_1, l_2$
form a $\Bbbk$-basis of $V$. Pick $l_3 \not \in V$, if such a linear form
exists, such that $R_1V = R_1V + R_1l_3$; otherwise, let $l_3 \not \in V$
be any linear form. Let $l_4$ be any linear form such that $l_1, \ldots,
l_4$ is a $\Bbbk$-basis for $R_1$. Choose $q_1, \ldots, q_4$ a
$\Bbbk$-basis of $R_1$ such that $\{q_1, q_2\}$ is a $\Bbbk$-basis of
$R_1l_1$ and $\{q_1, q_2, q_3\}$ is a $\Bbbk$-basis of $R_1V$.
{Let $s$ be
any generator of the socle of $R$ (which is a one-dimensional
$\Bbbk$-vector-space).}
Then $\Bbbk
l_1 \subsetneq \Bbbk \{l_1, l_2\} \subsetneq \Bbbk \{l_1, l_2, l_3\}
\subsetneq \Bbbk \{l_1, \ldots, l_4\}$, $\Bbbk q_1 \subsetneq \Bbbk \{q_1,
q_2\} \subsetneq \Bbbk \{q_1, q_2, q_3\} \subsetneq \Bbbk \{q_1, \ldots,
q_4\}$, $\Bbbk s$ is an embedding filtration of $R$.
\end{examplebox}

\begin{examplebox}[Strongly stable ideals]
If $\mathfrak a$ is a strongly stable $A$-ideal (also called
$0$-Borel) and $R = A/\mathfrak a$, then $\hilbertPoset_R$ need not admit
an embedding;  contrast this with \lex-segment ideals, in
Remark~\ref{remark:embInherits}. Take $A = \Bbbk[x_1, x_2, x_3]$, $\mathfrak a = (x_1,
x_2)^2(x_1, x_2, x_3)^2$. Consider $I = (x_1^2, x_1x_2, x_2^2)R$ and $I' =
(x_1^2, x_1x_2, x_1x_3)R$.  
Then $H_{I} = 3\hssymb^2 + 7 \hssymb^3$ and 
$H_{I'} = 3\hssymb^2 + 6\hssymb^3 + \sum_{t\geq 4} \hssymb^t$.
Hence, if $\hilbertPoset_R$ admitted an embedding, then there would be an
$R$-ideal $J$ with 
$H_{J} = 3\hssymb^2 + 6 \hssymb^3$;
however, such an ideal does not exist.
{(We see this as follows. Let $J$ be any homogeneous
$R$-ideal such that $|J_2| = 3$ and $|J_4| =0$. In particular,
$|J_2\mathfrak m^2| = 0$ so $J_2 \subseteq (0\!:_R\!\mathfrak m^2)_2 = I_2$.
Hence $J_2 = I_2$. Since $I$ is generated by $I_2$, $I_3 = I_2\mathfrak m$. 
Therefore $|J_3| \geq |I_3| = 7$.)}
\end{examplebox}

\begin{examplebox}[Gr\"obner flags]
Let $A = \Bbbk[x_1, \ldots, x_6]$ and $\mathfrak a = x_1(x_1, \dots, x_4) +
(x_2^2, x_2x_3) + (x_3^2) + x_4(x_4,x_5) + x_5(x_5,x_6)$. Let $R =
A/\mathfrak a$. 
Then $0 \subseteq \Bbbk\{x_1\} \subseteq \Bbbk\{x_1, x_2\} \subseteq
\cdots \subseteq R_1$ is a Gr\"obner flag for $R$; see~\cite{CRVgflags01} 
for the definition and properties. In particular $R$ is Koszul. Note that
$H_{(x_1)} = \hssymb^1 + 2 \hssymb^2 + \sum_{t=3}^\infty \hssymb^t$ and that 
$H_{(x_5)} = \hssymb^1 + 3\hssymb^2$.
There is, however, no $R$-ideal $I$ with
$H_{I} = \hssymb^1 + 2\hssymb^2$. 
(For, if $I$ is such that $|I_1| = 1$ and
$|I_3| = 0$, then $I \subseteq (0\!:_R\!\mathfrak m^2) = (x_5)$. Hence $I_1
= \Bbbk x_5$, and, therefore, $|I_2| \geq 3$.) 
\end{examplebox}

\begin{examplebox}[Tensor products] 
Let $A = \Bbbk[x,y,z]$ and $\mathfrak a = (x,y)^3 + (z^2)$. Then 
$H_{(x)} = \hssymb^1 + 3 \hssymb^2 + 2\hssymb^3$ and 
$H_{(z)} = \hssymb^1 + 2 \hssymb^2 + 3\hssymb^3$.
However, there is  no monomial $R$-ideal $I$ with
$H_{R/I} = \hssymb^1 + 2 \hssymb^2 + 2\hssymb^3$, for, if $I$ is any
monomial ideal with $I_1\not = 0$, then $H_I \succcurlyeq H_{(x)} =
H_{(y)}$ or 
$H_I \succcurlyeq H_{(z)}$.
Notice that $R \simeq \Bbbk[x,y]/(x,y)^3
\otimes_\Bbbk \Bbbk[z]/(z^2)$; both
$\hilbertPoset_{\frac{\Bbbk[x,y]}{(x,y)^3}}$ and 
$\hilbertPoset_{\frac{\Bbbk[z]}{(z^2)}}$ admit embedding. (It suffices to
consider monomial ideals; see Remark~\ref{remark:embImMonom}.)
\end{examplebox}

\begin{remarkbox}[Veronese subrings, due to A.~Conca]
Embeddings of Hilbert functions restrict to Veronese subrings. Let $R$ be a
standard graded algebra such that $\hilbertPoset_R$ admits an embedding
$\epsilon$.  Let $m \in \naturals$.
Let $S$ be the $m$th Veronese subring of $R$, i.e., $S = \bigoplus_{i
\in \naturals} S_i$ with $S_i = R_{im}$ for all $i \in \naturals$.
We define $\bar \epsilon : \hilbertPoset_S \longrightarrow \idealPoset_S$
as follows. Let $H \in \hilbertPoset_S$. Let
$\{0 = V_{d,0} \subsetneq V_{d,1} \subsetneq \cdots \subsetneq V_{d,|R_d|}
= R_d : d \in \naturals\}$ be the embedding filtration of $R$ given by
$\epsilon$. In degree $d \in \naturals$, we define $\bar\epsilon(H)$ by
setting $(\bar\epsilon(H))_d = V_{md, H^d}$. Note that 
$\bigoplus_{d \in \naturals} V_{md, H^d}$ 
is an ideal of $S$, since $R_m V_{md, H^d}
\subseteq V_{m(d+1), H^{d+1}}$. Hence 
$\{0 = V_{md,0} \subsetneq V_{md,1} \subsetneq \cdots 
\subsetneq V_{md,|R_{md}|} = S_d : d \in \naturals\}$ 
is an embedding filtration of $S$.
\end{remarkbox}

\begin{remarkbox}[Gotzmann property] 
\label{remarkbox:gotzmann}
Suppose that $R = A$, a polynomial ring. Then the Gotzmann Persistence
Theorem~\cite {BrHe:CM}*{Theorem~4.3.3} asserts that for all homogeneous
$R$-ideals $I$, if $I$  is generated minimally in degrees less than or
equal to $d$ and $|\mathfrak m I_d| \leq |\mathfrak mV|$ for all
$\Bbbk$-vector-spaces $V \subseteq R_d$ satisfying $|V| = |I_d|$, then for
all $t \geq d$, $|\mathfrak m I_t| \leq |\mathfrak mW|$ for all
$\Bbbk$-vector-spaces $W \subseteq R_t$ satisfying $|W| = |I_t|$. Now
suppose that $R$ is any standard graded algebra, such that
$\hilbertPoset_R$ admits an embedding $\epsilon$. Reinterpreting the
conclusion of the Gotzmann Persistence Theorem, we say that $R$ has the
\define{Gotzmann property} (with respect to $\epsilon$) if for all
homogeneous $R$-ideals $I$, if $I$ is minimally generated in degrees than
or equal to $d$ and $\epsilon(H_I)$  has no minimal generators in degree
$d+1$, then $\epsilon(H_I)$ is generated minimally in degrees less than or
equal to $d$ as well. 
Note that having the Gotzmann property is independent of the chosen
embedding: as we observed in Remark~\ref{remark:zerothBetti}, the degrees
and number of minimal generators of $\epsilon(I)$ depends only on $H_I$.

There are standard graded algebras $R$ such that $R$ does not have the
Gotzmann property while $\hilbertPoset_R$ admits an embedding. 
As an example, consider
$A=\Bbbk[x]$ and $\mathfrak a =(x^3)$. It is immediate that the 
quotient $R=A/\mathfrak a$ has the Gotzmann property and that
$\hilbertPoset_R$ admits an embedding.  Let 
$S=R[y]$ and notice that it has an embedding induced by lexicographic
order on $A[y]$. On the other hand $S$ fails to have the Gotzmann property 
since both $I=(y)S$ and  $\epsilon(I)=(x,y^3)S$ have no generators in
degree $2$. 
\end{remarkbox}

\begin{discussionbox}
\label{discussion:embedOrders}
Suppose that $R$ is defined by a monomial ideal or that it is
an affine semigroup algebra all of
whose generators are of the same degree.  Suppose that $R$ has an embedding
filtration.  Then, as in Remark~\ref{remark:embImMonom}, we can take
initial ideals, and obtain a graded total order $(\mathcal B, \tau)$ that
satisfies, for all $d \in \naturals$ and for all $\tau$-segments $V$ of
$R_d$,
\begin{inparaenum}
\item $R_1V$ is a $\tau$-segment of $R_{d+1}$, and,
\item $|R_1V| \leq |R_1W|$ for all $\Bbbk$-subspaces $W \subseteq R_d$ with
$|W| = |V|$.
\end{inparaenum}
We call such a graded total order an \define{embedding order} on $R$.
Conversely, any embedding order gives rise to an embedding filtration. We
thus conclude that $\hilbertPoset_R$ admits an embedding if and only if
there exists an embedding order on $R$.
\end{discussionbox}

When $R = A$ or $R = A/(x_1^2, \ldots, x_n^2)$ the embedding orders on $R$
are enumerated in~\cite {MermLexlike06}.

\begin{propn}
\label{thm:multOrderIsLex}
Let $(\mathcal B, \tau)$ be an embedding order of $R$. Suppose that it
is a monomial order.  
Then there exists a graded lexicographic order \lex on $A$ such that 
for all $f, f' \in \mathcal B$, if
$f \prec_\tau f'$, then $f \prec_\lex f'$.
\end{propn}

\begin{proof}
Without loss of generality, we may assume (since $\mathfrak{a}_1 = 0$) that 
$x_1 \prec_\tau \cdots \prec_\tau x_n$. Let $\lex$ be the graded
lexicographic order on $A$ with $x_1 \prec_\lex \cdots \prec_\lex x_n$.
Assume the contrary, and
let $t$ be the smallest degree such that there exist monomials 
$f$ and $f'$ in $A$ with $t = \deg f = \deg f'$ such that
$f \prec_\tau f'$ and $f' \prec_\lex f$. By the choice of $t$,
we see that $\gcd(f, f') = 1$.

Let $i$ be the smallest index such that $x_i$ divides at least one of $f$
and $f'$. By going modulo $(x_1,\dots,x_{i-1})$ and using
Remark~\ref{remark:embInherits}, we can assume, without loss of generality,
that $i=1$. Since $f' \prec_\lex f$, $x_1 \mid f'$.

Let $\{0 = V_{d,0} \subsetneq V_{d,1} \subsetneq \cdots \subsetneq
V_{d,|R_d|} = R_d : d \in \naturals\}$ be the embedding filtration of $R$
induced by $\tau$. 
Since $V_{1,1}$ is spanned by $\phi(x_1)$, there exists $s$
such that $V_{t,s}$ is the span of the images of all the
degree-$t$ monomials in $\mathcal B$ divisible by $x_1$. Therefore
$\phi(f') \in V_{t,s}$, which implies that 
$f' \prec_\tau f$, a contradiction.
\end{proof}

\begin{remarkbox}
If $R$ has an embedding order $\tau$, then for all $\tau$-segment
$R$-ideals $I$, $R/I$  inherits the embedding order $\tau$ (see
Remark~\ref{remark:embInherits}). 
Furthermore, $(0 \!:_R\!\mathfrak m)$ is a $\tau$-segment ideal. 
For $d \in \naturals$, let $s_d = \dim_\Bbbk (0 \!:_R\!\mathfrak m)_d$.
Let $\{0 = V_{d,0} \subsetneq V_{d,1} \subsetneq \cdots \subsetneq
V_{d,|R_d|} = R_d : d \in \naturals\}$ be the embedding filtration of $R$
induced by $\tau$. Then $|R_1 V_{d,s_d}| \leq |R_1(0 \!:_R\!\mathfrak m)_d| 
= 0$, so $V_{d,s_d} \subseteq (0 \!:_R\!\mathfrak m)_d$; since 
$|V_{d,s_d}| = |(0 \!:_R\!\mathfrak m)_d|$, we see further that
$V_{d,s_d} = (0 \!:_R\!\mathfrak m)_d$. Hence 
$(0 \!:_R\!\mathfrak m) = \oplus_d (0 \!:_R\!\mathfrak m)_d$ 
is a $\tau$-segment ideal.
\end{remarkbox}

\begin{examplebox}[Embedding with no monomial orders]
Let $A = \Bbbk[w, x, y, z]$, with homogeneous
maximal ideal $\mathfrak n$. Let
$\mathfrak a = (wxy, wxz, wyz, xyz) +
\mathfrak n^4$ and $R = A/\mathfrak a$. Without loss of generality, we
order $R_1$ by $w \prec_\tau x \prec_\tau y \prec_\tau z$. Note that
$|(w^2R)_3| = 4 > 2 = |(wxR)_3|$, so
$wx \prec_\tau w^2$. 
In fact, there is no embedding order in which $w^2$
precedes all the other monomials in $R_2$. On the other hand, we have an
embedding order $wx \prec_\tau wy \prec_\tau w^2 \prec_\tau wz \prec_\tau
xy \prec_\tau x^2 \prec_\tau xz \prec_\tau y^2 \prec_\tau yz \prec_\tau
z^2$.  Since $\mathfrak n^4 \subseteq \mathfrak a$ we can give any order on
$R_3$, provided $R_1V$ is a $\tau$-segment of $R_3$ for every
$\tau$-segment $V$ of $R_2$. For example, we may take $w^2x \prec_\tau wx^2
\prec_\tau w^2y \prec_\tau wy^2 \prec_\tau w^3 \prec_\tau w^2z \prec_\tau
wz^2 \prec_\tau x^2y \prec_\tau xy^2 \prec_\tau x^3 \prec_\tau x^2z
\prec_\tau xz^2 \prec_\tau y^3 \prec_\tau y^2z \prec_\tau yz^2 \prec_\tau
z^3$.\end{examplebox}

\section{Extension rings}
\label{sec:extnRings}
For certain $R$-free quotient rings $S$ of $R[z]$, 
we determine a sufficient condition (which we call \define{$z$-stability},
see Definition~\ref{defn:coeffSeq}) for extending an embedding filtration
of $R$ to $S$.  We use it to extend embeddings under polarization and
distraction, and prove the following analogue 
of a theorem of Clements and Lindstr\"om.

\begin{notation}
\label{notation:extnRings}
In this section $z$ is an indeterminate over $A$ and $R$.
Let $t \in \naturals \cup \{\infty\}$ and $S = R[z]/(z^t)$.
If $t = \infty$, the $R[z]$-ideal $(z^t)$ denotes the zero ideal.
As usual $R = A/\mathfrak a$, where $A$ is a polynomial ring. 
We will denote by $\mathfrak{I}$ the set $\{i\in \mathbb{Z} \vert 0\leq i <t\}.$
Let $B = A[z]$; treat $S$ as a quotient ring of $B$. Treat $B$ and $S$ as
\define{multigraded} using the natural decomposition $B = \oplus_{i \in
\naturals} \oplus_{j \in \naturals} A_iz^j$ as $\Bbbk$-vector-spaces. By 
$ \mathcal F := 
\{0 = V_{d,0} \subsetneq V_{d,1} \subsetneq \cdots 
\subsetneq V_{d,|R_d|} = R_d : d \in \naturals\}$, we mean the 
embedding filtration of $R$ 
that corresponds to the embedding of $\hilbertPoset_R$
(by Proposition~\ref{propn:embIFFfiltr}).
\end{notation}

\begin{defn}
\label{defn:coeffSeq}
Let $W \subseteq S_d$ be a multigraded $\Bbbk$-vector-space. The
\define{$R$-coefficient sequence} of $W$ is the sequence 
$(W_{d-i})_{i\in \mathfrak I}$ of $\Bbbk$-subspaces $W_{d-i} \subseteq R_{d-i}$
defined by the $\Bbbk$-vector-space decomposition $W =
\bigoplus_{i\in \mathfrak I}W_{d-i}z^i$.  We say that $W$ is
\define{$z$-stable} if $R_1 W_{d-i} \subseteq W_{d-i+1}$ for all positive  
$i\in \mathfrak I$. Let $I \subseteq S$ be an ideal; we say that $I$ is
\define{$z$-stable} if $I$ is multigraded and $I_d$ is $z$-stable for all
$d \geq 0$.
\end{defn}

\begin{thm}
\label{thm:extnRings}
Let $t \in \naturals \cup \{\infty\}$ and $S = R[z]/(z^t)$. Suppose that
$\hilbertPoset_R$ admits an embedding and that for all $H \in
\hilbertPoset_S$, there exists a $z$-stable $S$-ideal $I$ such that
$H_{I} = H$. Then $\hilbertPoset_S$ admits an embedding.
\end{thm}

\begin{defn}
Let $d \in \naturals$. A \define{segment} of $S_d$ is a $z$-stable  $\Bbbk$-vector-space
$\bigoplus_{i\in \mathfrak I}V_{d-i,r_{d-i}}z^i$ such that
$V_{d-i, r_{d-i}} \subseteq R_{j-i}V_{d-j,\min\{1+r_{d-j}, |R_{d-j}|\}}$, 
for all $i < j \in \mathfrak I$;
its \define{length} is $\sum_{i\in \mathfrak I} r_{d-i}$.
\end{defn}

\begin{obs}
\label{obs:zStableGrowth} 
Let $W \subseteq S_d$ be a $z$-stable $\Bbbk$-vector-space, with
$R$-coefficient sequence $(W_{d-i})_{i\in \mathfrak I}$.
Then 
\[
S_1W = 
R_1W_d \bigoplus \bigoplus_{\substack{i \in \mathfrak I \\ i > 0}}
W_{d-i+1}z^i
\]
If, further, $W$ is a segment of $S_d$, $S_1W$ is a segment of $S_{d+1}$.
\end{obs}

\begin{defn}
For a multigraded $\Bbbk$-vector-space $W \subseteq S_d$ with
$R$-coefficient sequence $(W_{d-i})_{i\in \mathfrak I}$,
let 
\[
d_R(W) = \left(\sum_{j=0}^{i}|W_{d-j}|\right)_{i\in \mathfrak I}.
\]
Let $\Lambda_d = \{d_R(W) : W \;\text{is a 
multigraded $\Bbbk$-vector-space of $S_d$}\}$.
Give a partial order $\lessdot$ on $\Lambda_d$ by setting 
$(a_i)_{i\in \mathfrak I} \lessdot (b_i)_{i\in \mathfrak I}$ 
if $a_i \leq b_i$ for all $i$. For all 
$(a_i)_{i\in \mathfrak I} \in \Lambda_d$, $a_i = a_d$
for all $i \geq d$ and $a_d \leq |S_d|$; hence $\Lambda_d$ is a finite set.
\end{defn}

\begin{lemma}
\label{lemma:zStableMinSegment}
Let $W = \bigoplus_{i\in \mathfrak I}V_{d-i,r_{d-i}}z^i
\subseteq S_d$ be a $z$-stable $\Bbbk$-vector-space such that
$d_R(W)$ is minimal in $(\Lambda_d, \lessdot)$. Then $W$ is a segment.
\end{lemma}

\begin{proof}
By way of contradiction assume that $W$ is not a segment. 
Pick $i < j \in \mathfrak I$ such that 
$V_{d-i, r_{d-i}} \not \subseteq 
R_{j-i}V_{d-j,\min\{1+r_{d-j}, |R_{d-j}|\}}$. 
We may assume that $j-i$ is minimal with this property.
As $V_{d-j, |R_{d-j}|} = R_{d-j}$, we see that $r_{d-j} < |R_{d-j}|-1$. 
Hence $V_{d-i, r_{d-i}} \not \subseteq R_{j-i}V_{d-j,1+r_{d-j}}$; since
these vector-spaces belong to $\mathcal F$, we observe that
\begin{equation}
\label{equation:Wzstablecondition}
R_{j-i}V_{d-j,1+r_{d-j}} \subsetneq V_{d-i, r_{d-i}}.
\end{equation}
Define
\[
\tilde W = \bigoplus_{h\in \mathfrak I \minus \{i,j\}}V_{d-h,r_{d-h}}z^h 
\bigoplus V_{d-i, r_{d-i}-1}z^i
\bigoplus V_{d-j, r_{d-j}+1}z^j.
\] 
It is immediate that $d_R(\tilde W) \lessdot d_R(W)$ and that they are not
equal to each other. Hence it suffices to show that $\tilde W$ is 
$z$-stable. We consider two cases.

If $j=i+1$, then we need to show that
$R_1V_{d-i-1, 1+ r_{d-i-1}} \subseteq V_{d-i, r_{d-i}-1}$, which is
immediate from~\eqref{equation:Wzstablecondition}.
Otherwise, we first show that, for all $i'$, $i < i' < j$, 
$V_{d-i', r_{d-i'}} = R_{j-i'} V_{d-j, 1 + r_{d-j}}$. 
Both the vector-spaces belong to $\mathcal F$, so they are comparable 
(with respect to inclusion). Minimality of $j-i$  implies 
$V_{d-i', r_{d-i'}} \subseteq R_{j-i'} V_{d-j, 1 + r_{d-j}}$. 
If $V_{d-i', r_{d-i'}} \subsetneq R_{j-i'} V_{d-j, 1 + r_{d-j}}$, then 
$V_{d-i', 1+ r_{d-i'}} \subseteq R_{j-i'} V_{d-j, 1 + r_{d-j}}$, so,
using~\eqref{equation:Wzstablecondition}, we see that
$R_{i'-i}V_{d-i', 1+ r_{d-i'}} \subseteq R_{j-i} V_{d-j, 1 + r_{d-j}}
\subsetneq V_{d-i, r_{d-i}}$, contradicting the
minimality of $j-i$. Taking $i' = i+1$ and $i' = j-1$, now, completes the
proof of the assertion that $\tilde W$ is $z$-stable.
\end{proof}

\begin{lemma}
\label{lemma:segsLengthComparison}
Let $W, W'$ be segments in $S_d$. If $|W| \leq |W'|$, then $W \subseteq W'$.
In particular, for every $1 \leq s \leq |S_d|$, there exists a unique segment
of length $s$.
\end{lemma}

\begin{proof}
Write $W = \bigoplus_{i \in \mathfrak I}V_{d-i,r_{d-i}}z^i$ and 
$W' = \bigoplus_{i \in \mathfrak I}V_{d-i,s_{d-i}}z^i$. 
We need to show that 
$r_{d-i} \leq s_{d-i}$ for all $i \in \mathfrak I$. 
Assume, by way of contradiction, 
that there exists $i$ such that $r_{d-i} > s_{d-i}$. Then we
observe that $1+s_{d-i} \leq |R_{d-i}|$ 
and, hence, that
\[
V_{d-i, s_{d-i'}} \subseteq R_{i-i'}V_{d-i,1+s_{d-i}} \subseteq
R_{i-i'}V_{d-i,r_{d-i}} \subseteq V_{d-i', r_{d-i'}},
\]
for all $i' < i$. Hence $s_{d-i'} \leq r_{d-i'}$ for all $i' < i$.
However, since $\sum_{j \in \mathfrak I} r_{d-j} =
\sum_{\substack{j \in \mathfrak I \\ j \leq d}} r_{d-j} \leq
\sum_{\substack{j \in \mathfrak I \\ j \leq d}} s_{d-j} =
\sum_{j \in \mathfrak I} s_{d-j}$,
there exists $j > i$ such that $s_{d-j} > r_{d-j}$. 
Repeating the above argument, now reversing the roles of $W$ and $W'$, we
see that $s_{d-i} > r_{d-i}$, contradicting our assumption.
Hence $r_{d-i} \leq s_{d-i}$ for all $i \in \mathfrak I$. 

Let $1 \leq s \leq |S_d|$. The existence of a segment of length $s$ follows
from Lemma~\ref{lemma:zStableMinSegment}. Uniqueness is immediate from the
first assertion of this lemma.
\end{proof}

\begin{proof}[Proof of Theorem~\protect{\ref{thm:extnRings}}]
Let $H \in \hilbertPoset_S$. Let $I$ be a $z$-stable $S$-ideal such
that $H_{I} = H$. Let $d \in \naturals$. Let $J_d$ be the (unique) segment
of $S_d$ of length $|I_d|$, which exists by 
Lemma~\ref{lemma:segsLengthComparison}.
By Observation~\ref{obs:zStableGrowth},
$S_1J_d$ is a segment of $S_{d+1}$. 
By the minimality of $d_R(J_d)$ (among all the $z$-stable subspaces of
$S_d$ of length $|I_d|$) and Observation~\ref{obs:zStableGrowth}, we see
that 
$|S_1J_d| \leq |S_1I_d| \leq |I_{d+1}| = |J_{d+1}|$
and that $S_1J_d$ is a segment of $S_{d+1}$. 
Hence, by Lemma~\ref{lemma:segsLengthComparison}, $S_1J_d
\subseteq J_{d+1}$. Therefore $J = \bigoplus_d J_d$ is an $S$-ideal. Now
the map $\epsilon : \hilbertPoset_S \longrightarrow \idealPoset_R$ sending
$H \mapsto J$ is an embedding.
\end{proof}

We note that Lemmas~\ref{lemma:zStableMinSegment}
and~\ref{lemma:segsLengthComparison} makes no reference to the hypothesis
of Theorem~\ref{thm:extnRings}. Consequently, applying them and arguing as
in the proof of Theorem~\ref{thm:extnRings}, we obtain the following:

\begin{thm}
\label{theorem:strongHyp}
Let $t \in \naturals \cup \{\infty\}$ and $S = R[z]/(z^t)$. Suppose that
$\hilbertPoset_R$ admits an embedding $\epsilon$. Let $I$ be a $z$-stable
$S$-ideal and $J = \epsilon(H_I)$. Then
$H_{(I,z^i)} \succcurlyeq H_{(J,z^i)}$ for all $i \in \mathfrak I$.
\end{thm}

\begin{proof}
Let $I$ be a $z$-stable $S$-ideal such that $H_{I} = H$. We construct the
ideal $J$ by minimizing the function $d_R(J_d)$ for all $d$ as in the proof
of Theorem~\ref{thm:extnRings}; see the proof of
Lemma~\ref{lemma:zStableMinSegment}. Note that $J$
depends only on $H$. Let $i \in \mathfrak I$ and $d \in \naturals$.
We want to show that $|(I,z^i)_d| \geq |(J,z^i)_d|$. 
Let $(W_{d-j})_{j\in \mathfrak I}$ and
$(V_{d-j,r_{d-j}})_{j\in \mathfrak I}$ be, respectively,
the $R$-coefficient sequences of $I_d$ and $J_d$.
Then $|(I,z^i)_d| \geq |(J,z^i)_d|$ if and only if 
\[
\sum_{\substack{j \in \mathfrak I \\ j < i}} |W_{d-j}| \geq
\sum_{\substack{j \in \mathfrak I \\ j < i}} r_{d-j}
\]
which, indeed, is true by the construction of $J_d$.
\end{proof}

In Discussion~\ref{discussion:embedOrders} we noted that if 
$R$ is defined by a monomial ideal that it is
an affine semigroup algebra all of whose
generators are of the same degree, $\hilbertPoset_R$ admits an embedding if
and only if there exists an embedding order on $R$. In this situation, we
can strengthen Theorem~\ref{thm:extnRings},  to conclude 
that if the embedding order on $R$ is induced by a graded lexicographic 
order on $A$ (in the sense of Proposition~\ref{thm:multOrderIsLex}), 
then the extension to
$S$ is induced by a graded lexicographic order on $B$.
\begin{defn}
\label{defn:extnRingsEmbOrder}
Let $(\mathcal B, \sigma)$ be an embedding order on $R$.
Let $\mathcal B' = \{fz^i : f \in \mathcal B, i\in \mathfrak I\}$.
Define a total order $\tau$ on $\mathcal B'$ as follows.
Let $fz^a, gz^b \in \mathcal B'$ with $f, g \in \mathcal B$, $\deg f + a =
\deg g + b$ and $a \leq b$.
Set 
$fz^a \prec_\tau gz^b$ if there exists $g' \in \mathcal B$ with  $\deg g' = \deg g$
such that
$f \in A_{b-a}g'$ and  $g' \prec_\sigma g$. Otherwise set $gz^b   \prec_\tau  fz^a.$
\end{defn}

\begin{thm}
\label{thm:extnRingsEmbOrder}
With notation as above, $(\mathcal B', \tau)$ is an embedding order for
$S$; the embedding of $\hilbertPoset_S$ from Theorem~\ref{thm:extnRings} is
induced by $(\mathcal B', \tau)$. Moreover, if $x_1, \ldots, x_n$ are the
variables of $A$ and $(\mathcal B, \sigma)$ is a monomial order with $x_1
\prec_\sigma \cdots \prec_\sigma x_n$ then $(\mathcal B', \tau)$ is a
monomial order with $x_1 \prec_\tau \cdots \prec_\tau x_n \prec_\tau z$.
\end{thm}

\begin{proof}
Indeed, $\mathcal B'$ is a standard basis for $S$.
To show that $(\mathcal B', \tau)$ is an embedding order, it suffices to
show that $(\mathcal B', \tau)$ induces the embedding from 
Theorem~\ref{thm:extnRings}.

Recall that $B = A[z]$; write $\phi$ for the surjective 
homomorphism $B \to S$.
Consider $f, g \in \mathcal B, f \neq g$ and $0 \leq a \leq b < t$ with $\deg
f + a = \deg g + b = d$. 
By Lemma~\ref{lemma:segsLengthComparison}, there exist 
segments $W \subsetneq W' \subseteq S_d$ such that 
$W'=W+ \Bbbk \cdot \langle \phi(gz^b) \rangle$.
As in Definition~\ref{defn:coeffSeq}
write $W = \bigoplus_{i\in \mathfrak I}V_{d-i,r_{d-i}}z^i$ 
(with $R_{j-i}V_{d-j,r_{d-j}} \subseteq V_{d-i, r_{d-i}} \subseteq
R_{j-i}V_{d-j,1+r_{d-j}}$, for all $0 \leq i < j < t$
).
Hence 
\[W' = \bigoplus_{i=0}^{b-1}V_{d-i,r_{d-i}}z^i \bigoplus
V_{d-b,1+r_{d-b}}z^b \bigoplus 
\bigoplus_{\substack{i \in \mathfrak I \\ i > b}}
V_{d-i,r_{d-i}}z^i.
\]
Since $W'$ is a segment, we see that 
$V_{d-i,r_{d-i}} = R_{b-i}V_{d-b,1+r_{d-b}}$ for all $0 \leq i \leq b-1$
and that $V_{d-b, 1 + r_{d-b}} \subseteq 
R_{j-b}V_{d-j,1+r_{d-j}}$, for all $b+1 \leq j < t$.
Note that 
$V_{d-b,1+r_{d-b}} = \Bbbk \cdot \langle \phi(g) \rangle \bigoplus V_{d-b,r_{d-b}}$.

Since 
$a \leq b$, $\phi(fz^a) \in W$ if and only if $f \in A_{b-a}g'$ for
some monomial $g' \in \mathcal B$ with $\deg g' = \deg g$ and $g'
\prec_\sigma g$, i.e., if and only if $fz^a \prec_\tau gz^b$, so the
embedding of $\hilbertPoset_S$ is induced by $\tau$.
Now, in order to prove that $(\mathcal B', \tau)$ is a monomial order if
$(\mathcal B, \sigma)$ is, consider, as above, 
$f, g \in \mathcal B$ and $0 \leq a \leq b < t$ with $\deg f + a
= \deg g + b = d.$ When  $fz^a \prec_\tau gz^b,$
$h \in \mathcal B$ and $0 \leq c+a < t$, 
it follows directly from Definition~\ref{defn:extnRingsEmbOrder} that 
$fhz^{a+c} \prec_\tau
ghz^{b+c}$. On the other hand if $gz^b \prec_\tau fz^a$ we know that whenever 
$g' \in \mathcal B$ with  $\deg g' = \deg g$ and $f \in A_{b-a}g'$ we also must have  $g \prec_\sigma g',$ hence  
$ghz^{b+c} \prec_\tau fhz^{a+c}.$
\end{proof}

\subsection*{Distractions}

A distraction is a $\Bbbk$-linear automorphism $\phi$ of a polynomial ring
that has the property that if $m$ and $n$ are monomials and $m$ divides
$n$, , then $\phi(m)$ divides $\phi(n)$.
We follow the formulation of \cite{BCRdistr05}, and show that embedding
the Hilbert functions can be extended to distractions. (For earlier work
using distractions, see~\cite{BCRdistr05}*{Introduction}.) We consider a
special case, and describe how embeddings of Hilbert series can be extended
to distractions.

\begin{notation}
Let $\Bbbk$ be an infinite field and $A = \Bbbk[x_1, \ldots, x_n]$, with
$\naturals^2$-grading given by $\deg x_1 = (1,0)$ and $\deg x_i = (0,1)$,
for all $2 \leq i \leq n$. Let $\mathfrak a$ be a multigraded ideal.
As earlier, $R = A/\mathfrak a$.
\end{notation}
    
\begin{defn}[\protect\cite{BCRdistr05}*{Definitions~2.1,~2.2}] 
\label{defn:distraction}
Let $N \in \naturals$. A \define{distraction matrix} is an infinite matrix $L
= \left(l_{i,j}: 1 \leq i \leq n, j \in \naturals \right)$ such that 
\begin{inparaenum}
\item $l_{i,j} \in A_1$ for all $i,j$,
\item for all $j_1, \ldots, j_n \in \naturals$, $\{l_{1,j_1}, \ldots,
l_{n,j_n}\}$ spans $A_1$ (as a $\Bbbk$-vector-space), and,
\item there exists $N \in \naturals$ such that $l_{i,j} = l_{i,N}$ for all
$j \geq N$.
\end{inparaenum}
The \define{distraction} associated to $L$, $D_L$, is the
$\Bbbk$-vector-space morphism $D_L : A \longrightarrow A$ such that
$\prod_{i=1}^n x_i^{a_i} \mapsto \prod_{i=1}^n \prod_{j=1}^{a_i} l_{i,j}$.
\end{defn}

\begin{remark}
\label{remark:distractionPreservesDirectSums}
We observe that, for all $d \in \naturals$, 
$D_L|_{A_d} : A_d \to A_d$ is an isomorphism of vector spaces. Therefore
for all subspaces $V, V'$ of $A_d$ with $V \cap V' = 0$, $D_L(V+V') =
D_L(V) \oplus D_L(V')$. In particular, $D_L$ preserves Hilbert functions.
\end{remark}

In~\cite{BCRdistr05}, the authors consider monomial ideals. We, however,
need the results in bigraded ideals. The following lemma can be proved
following their ideas, but for the sake of completeness, we include a
proof. We will use this lemma, again, in the proof of
Theorem~\ref{thm:polzEmb}.

\begin{lemma}
\label{lemma:bigradedDistr}
Let $\mathfrak a$ be a multigraded ideal, and $L = \left(l_{i,j}: 1 \leq i
\leq n, j \in \naturals \right)$ be a distraction matrix such that $l_{i,j}
= x_i$ for all $2 \leq i \leq n$ and for all $j \in \naturals$. Then
$D_L(\mathfrak a)$ is an ideal.
\end{lemma}

\begin{proof}
Let $\mathfrak a_{(j)} = ((I \!:_A x_1^j) \cap \Bbbk[x_2, \ldots x_n]), j
\in \naturals$. They are ideals in $\Bbbk[x_2, \ldots x_n]$, and,
as $\Bbbk$-vector-spaces, $\mathfrak a = \bigoplus_{j \in
\naturals} \mathfrak a_{(j)}x_1^j$.
Hence, by Remark~\ref{remark:distractionPreservesDirectSums},
$D_L(\mathfrak a) = \bigoplus_{j \in
\naturals} \left(\mathfrak a_{(j)}\prod_{t=1}^j l_{1,t}\right)$. 
Since the $\mathfrak a_{(j)}$ are ideals, note that $x_i D_L(\mathfrak a)
\subseteq D_L(\mathfrak a)$ for $2 \leq i \leq n$.
To finish the proof, it suffices to show
that $x_1 \mathfrak a_{(j)}\prod_{t=1}^j l_{1,t} \subseteq 
D_L(\mathfrak a)$. Let $f \in \mathfrak a_{(j)}$. We want to show that 
$fx_1\prod_{t=1}^j l_{1,t} \in D_L(\mathfrak a)$. Since 
$\mathfrak a_{(j)} \subseteq \mathfrak a_{(j+1)}$, we know that 
$f \prod_{t=1}^{j+1} l_{1,t} \in D_L(\mathfrak a)$. Moreover, 
$f x_i \prod_{t=1}^{j} l_{1,t} \in D_L(\mathfrak a)$
for all $2 \leq i \leq n$ (since 
$x_i\mathfrak a_{(j)}  \subseteq \mathfrak a_{(j)}$). Therefore
$fx_1\prod_{t=1}^j l_{1,t} \in D_L(\mathfrak a)$. 
\end{proof}

\begin{propn}
\label{thm:distrEmb}
Let $\mathfrak a$ be a multigraded ideal, and $L = \left(l_{i,j}: 1 \leq i
\leq n, j \in \naturals \right)$ be a distraction matrix such that $l_{i,j}
= x_i$ for all $2 \leq i \leq n$ and for all $j \in \naturals$.  If
$\hilbertPoset_R$ admits an embedding, then
$\hilbertPoset_{\frac{A}{D_L(\mathfrak a)}}$ admits an embedding.
\end{propn}

\begin{proof}
(Indeed, by Lemma~\ref{lemma:bigradedDistr}, $D_L(\mathfrak a)$ is an
ideal.)
Write $S = \frac{A}{D_L(\mathfrak a)}$.
Let $\omega$ be the weight order with $w(x_1) = 1$ and $w(x_i) = 0$ for all
$2 \leq i \leq n$.
Then $\In_\omega\left(D_L(\mathfrak a)\right) =
\mathfrak a$. (To see this, note that it is enough to show that 
$\mathfrak a \subseteq \In_\omega\left(D_L(\mathfrak a)\right)$. Let $f \in 
\mathfrak a$ be a multigraded element with $\deg f = (a,b)$. 
Write $f = x_1^ag$, with $g$ a homogeneous polynomial of degree $b$ in
$x_2, \ldots, x_n$. Then $D_L(f) = (\prod_{j=1}^a l_{1,j}) \cdot g$. 
Note that $x_1^a$ appears with a non-zero coefficient in 
$\prod_{j=1}^a l_{1,j}$. Hence
$f = \In_\omega(D_L(f)) \in \In_\omega\left(D_L(\mathfrak a)\right)$.)
Let $H \in \hilbertPoset_S$. 
Let $I \in \idealPoset_{A}$ be such that
$D_L(\mathfrak a) \subseteq I$ and 
$H_{IS} = H$. Then $\mathfrak a \subseteq \In_\omega(I)$.
Define $\epsilon' : \hilbertPoset_{S} \longrightarrow 
\idealPoset_{A/ D_L(\mathfrak a)}$
by sending 
$H \mapsto (D_L(\epsilon(\In_\omega(I))))S$; this is an embedding.
\end{proof}

\begin{remarkbox}
The same proof will work for a more general distraction matrix, in which,
for all $i$ and $j$, $l_{i,j}$ is a linear form in $x_i, x_{i+1}, \ldots,
x_n$, with $x_i$ appearing with a non-zero coefficient. However, unlike
polarization (discussed below), where working with one variable generalizes
to the general case, the distraction matrix in
Proposition~\ref{thm:distrEmb} is not general.
\end{remarkbox}

\subsection*{Polarization}
\label{sec:polarization}

We use Theorem~\ref{thm:extnRings} to show that polarization preserves
embeddability of Hilbert functions. Let $A = \Bbbk[x_1, \ldots, x_n, y]$ be
a polynomial ring.  Polarization~\cite{MiStCCA05}*{Exercise~3.15} is an
operation that converts an $A$-ideal $\mathfrak a$ to an $A[z]$-ideal
$\mathfrak b$. We will show that every embedding of
$\hilbertPoset_{A/\mathfrak a}$ gives rise to an embedding of
$\hilbertPoset_{A[z]/\mathfrak b}$.  Any polarization can be achieved by
repeatedly applying partial polarizations, so we will restrict our
discussion to this case.

\begin{notation}
\label{notation:polazSecNotation}
Let $A$ and $z$ be as above. We give $A$ the $\naturals^2$-grading with
$\deg x_i = (1,0)$, for all $i$, and $\deg y = (0,1)$. 
Write $B = A[z]$, graded with $\deg x_i = (1,0,0)$ for all $i$, $\deg y =
(0,1,0)$ and $\deg z = (0,0,1)$. 
Homogeneous
elements, ideals and modules in these gradings will be referred to as
\define{multigraded}. For a multigraded element $f$ (of $A$ or $B$), we
will denote its degrees by $\deg_{\mathbf x} f$, $\deg_y f$ and $\deg_z f$. 
By $\mathfrak a$, we will mean a multigraded $A$-ideal.
\end{notation}

\begin{defn}
\label{defn:polzn}
A \define{polarization} is a $\Bbbk$-vector-space morphism $p_{y,d,z}^A : A
\longrightarrow A[z]$, for some $d \in \naturals$, such that, for all
homogeneous forms $f \in A$ with $\deg_y f= 0$,
\[
p_{y,d,z}^ A: fy^i \mapsto
\begin{cases}
fy^i, & \text{if}\; i < d,  \\
fy^{i-1}z, & \text{otherwise}.
\end{cases}
\]
\end{defn}
\begin{remarkbox}
\label{remark:polznProps}
Let $\mathfrak a$  be a multigraded $A$-ideal, and let 
$\mathfrak b = (p_{y,d,z}^A(\mathfrak a))A[z]$. Then $\mathfrak b $ is a
multigraded $B$-ideal. Moreover, $y-z$ is a non-zero-divisor on
$B/\mathfrak b$. Hence $H_{A/\mathfrak a}(\hssymb) =
(1-\hssymb)H_{B/\mathfrak b}(\hssymb)$.
\end{remarkbox}

\begin{lemma}
\label{lemma:polznInitIdeal}
Let $\omega$ be the weight vector on $B$ with $\omega(x_i) = 1$ for all
$i$, $\omega(y) = 1$ and $\omega(z) = 0$. Let $g : B \longrightarrow B$ be the
$\Bbbk$-algebra morphism induced by the change of coordinates $x_i \mapsto
x_i$ for all $i$, $y \mapsto y$ and $z \mapsto y+z$.
Then $\In_\omega\left(g( p_{y,d,z}(\mathfrak a))\right) = \mathfrak a B$.
\end{lemma}

\begin{proof}
It follows from the definition that $\mathfrak aB \subseteq
\In_\omega\left(g( p_{y,d,z}(\mathfrak a))\right)$. Observe that the
Hilbert series of $p_{y,d,z}(\mathfrak a)$ and of 
$\In_\omega\left(g( p_{y,d,z}(\mathfrak a))\right)$ are identical. 
It is easy to see, from
Remark~\ref{remark:polznProps}, that the Hilbert series of $\mathfrak aB$
and of $p_{y,d,z}(\mathfrak a)$ are identical. Therefore $\mathfrak aB =
\In_\omega\left(g( p_{y,d,z}(\mathfrak a))\right)$.
\end{proof}
\begin{thm}
\label{thm:polzEmb}
If $\hilbertPoset_R$ admits an embedding, then
$\hilbertPoset_{B/(p_{y,d,z}^A(\mathfrak a))}$ admits an embedding.
\end{thm}

\begin{proof}
Write $S = B/(p_{y,d,z}^A(\mathfrak a))$.
First, for every homogeneous $B$-ideal $I$ containing $\mathfrak aB$, there
exists a homogeneous $B$-ideal $J$ containing $\mathfrak aB$ such that
$J(B/\mathfrak aB)$ is $z$-stable; see Lemma~\ref{lemma:stabilization}.
Now, by Theorem~\ref{thm:extnRings},
there is an embedding 
$\epsilon' : \hilbertPoset_{B/\mathfrak aB} \longrightarrow
\idealPoset_{B/\mathfrak aB}$.
Let $H \in \hilbertPoset_S$. Let $I$ 
be an $B$-ideal such that  $p_{y,d,z}^A(\mathfrak a) \subseteq I$ and $H =
H_{IS}$. Applying Lemma~\ref{lemma:polznInitIdeal}, we find $I' \subseteq
B$ such that $\mathfrak aB \subseteq I'$ and $H_{I'(B/\mathfrak aB)} = H$. 
We may assume that $I'(B/\mathfrak aB) = \epsilon'(H)$. 
Taking the initial ideal with respect to a
suitable weight order, we may further assume that
$I'$ is multigraded (in the grading of $B$).

Let $L$ be the following distraction matrix
(Definition~\ref{defn:distraction}): 
\[
\bordermatrix{ & 1 & 2 & \cdots & d-1 & d & d+1 & \cdots \cr
x_1 & x_1 & x_1 & \cdots & x_1 & x_1 & x_1 & \cdots \cr
\vdots & \vdots & \vdots & \ddots & \vdots & \vdots & \vdots & \ddots \cr
x_n & x_n & x_n & \cdots & x_n & x_n & x_n & \cdots \cr
y & y & y & \cdots & y & y+z & y & \cdots \cr
z & z & z & \cdots & z & z & z & \cdots \cr
}
\]
By Lemma~\ref{lemma:bigradedDistr} we see that both $D_L(\mathfrak aB)$ and
$D_L(I')$ are $B$-ideals; additionally, 
$D_L(\mathfrak aB) \subseteq D_L(I')$.
Let $\omega$ be a weight order with $\omega(x_i) = 1$ for all $i$, $\omega(y) =
0$ and $\omega(z) = 0$.
Then $\In_\omega\left(D_L(\mathfrak aB)\right) = 
 p_{y,d,z}^A(\mathfrak a)$.
Define $\epsilon : \hilbertPoset_S \longrightarrow \idealPoset_S$ by setting
$\epsilon : H \mapsto \In_\omega(D_L(I'))S$.
\end{proof}

\begin{remarkbox}
Mermin showed that if a monomial complete intersection $R = A/\mathfrak a$
(i.e., $\mathfrak a$ is generated by an $A$-regular sequence of
monomials) has the property that every Hilbert function is attained by the
image of a \lex-segment ideal, then $\mathfrak a = (x_1^{e_1}, \ldots,
x_r^{e_r-1}x_i)$ for some $e_1 \leq \cdots \leq e_r$ and $i \geq
r$~\cite{MerminMonomCI10}*{Theorem~4.4}.  Theorem~\ref{thm:polzEmb} shows
that if we allow for other graded term orders, then $\hilbertPoset_R$
admits an embedding for every monomial complete intersection $R$.
\end{remarkbox}

\subsection*{A Clements--Lindstr\"om type theorem for embeddings}
\label{sec:ClemLind}

We prove an analogue of the following theorem of Clements and Lindstr\"om:
If $A = \Bbbk[x_1, \ldots, x_n]$ and $\mathfrak a = (x_1^{e_1}, \ldots,
x_n^{e_n})$ with $2 \leq e_1 \leq \cdots \leq e_n \leq \infty$, then for
every homogeneous $A$-ideal $I$ with $\mathfrak a \subseteq I$,
there exists a \lex-segment ideal $L$ such that Hilbert functions of $L +
\mathfrak a$ and $I$ are identical. If $t = \infty$, then the theorem below
(even without the hypothesis that $\mathfrak a$ is a monomial ideal) 
follows from the argument that proved the
existence of the embedding $\epsilon'$ in the beginning of the proof of
Theorem~\ref{thm:polzEmb}. Note that, in that context, we may take
$\mathfrak a$ to be homogeneous in the standard grading of $A$ to apply
Lemma~\ref{lemma:stabilization} and Theorem~\ref{thm:extnRings}.

\begin{thm}
\label{thm:ClemLindEmb}
Let $A = \Bbbk[x_1, \ldots, x_n]$ and $t \in \ints \cup \{\infty\}$. Let
$\mathfrak a$  be a monomial $A$-ideal such that $x_i^t \in
\mathfrak a$ for all $1 \leq i \leq n$. Let $B = A[z]$, where $z$ is an
indeterminate. If $\hilbertPoset_R$ admits an embedding, then
$\hilbertPoset_{B/(\mathfrak aB,z^t)}$ admits an embedding.
\end{thm}

\begin{proof}
Since $\mathfrak a$ is a monomial ideal, in order to to study embeddings,
we need to consider only monomial ideals (Remark~\ref{remark:embImMonom});
hence we may assume that $\Bbbk = \complex$. 
Write $S = B/(\mathfrak aB, z^t)$. 
For every $H \in \hilbertPoset_S$, there exists a $z$-stable $S$-ideal $I$
such that $H_{I} = H$;~see Lemma~\ref{lemma:stabilizationCharZero}.
By Theorem~\ref{thm:extnRings}
$\hilbertPoset_S$ admits an embedding.
\end{proof}

\section{Stabilization}
\label{sec:stabilization}
The results of this section do not depend on the previous sections, and are
used in Section~\ref{sec:extnRings}.

We adopt the following notation for this section.  Let $A = \Bbbk[x_1,
\ldots, x_n]$ be a standard graded polynomial ring.  Let $\mathfrak a$ be
an $A$-ideal. Let $B = A[z]$, with $\naturals^2$-grading given by $\deg x_i
= (1,0)$, for all $1 \leq i \leq n$ and $\deg z = (0,1)$.  Let $\omega$ be
the weight vector on $B$ with $\omega(x_i) = 1$ for all $i$ and $\omega(z)
= 0$. 

\begin{lemma}
\label{lemma:stabilization}
Let $I$ be an $B$-ideal such that  $\mathfrak aB \subseteq I$. Then there
exists $B$-ideal $J$ such that $\mathfrak aB \subseteq J$, $H_{I} =
H_{J}$ and $J(B/\mathfrak aB)$ is $z$-stable.
\end{lemma}
\begin{proof}
Write $S = B/\mathfrak aB$.
Let $\omega$ be the weight vector on $B$ with $\omega(x_i) = 1$ for all
$i$ and $\omega(z) = 0$. 
By replacing $I$ by $\In_\omega(I)$, we may assume 
that $I$ is a multigraded $B$-ideal.

For $1 \leq l \leq n$, let $L_l$ be the distraction matrix
\[
\bordermatrix{ & 1 & 2 & \cdots &  t & t+1 & \cdots \cr
x_1 & x_1 & x_1 & \cdots & x_1 &  x_1 & \cdots \cr
\vdots & \vdots & \vdots & \ddots & \vdots & \vdots & \ddots \cr
x_n & x_n & x_n & \cdots & x_n &  x_n & \cdots \cr
z & x_l+z &  z & \cdots &  z & z & \cdots \cr
}.
\]
Write $\Phi = \In_\omega \circ D_{L_1} \circ \cdots \circ 
\In_\omega \circ D_{L_n}$.
Let $J^{(r)} = \Phi^r(I)$. 
We will show that $J^{(r)} = J^{(r+1)}$ for all $r \gg 0$. 
Note that 
$\Phi(\mathfrak aB) = \mathfrak aB$.
Write $J^{(r)} = \bigoplus_{d \in \naturals}\bigoplus_{i \in
\naturals}J^{(r)}_{d,d-i}z^i$ as $\Bbbk$-vector-spaces. For all $d \in
\naturals$ and all $i \in \naturals$ (equivalently, $0 \leq i \leq d$), we
have $\sum_{j=0}^i |J^{(r+1)}_{d,d-j}| \geq \sum_{j=0}^i
|J^{(r)}_{d,d-j}|$; equality holds for all $i \in \naturals$ 
(equivalently, $0 \leq i \leq d$)
if and only if
$J^{(r)}_d$ is $z$-stable.

Since $|J^{(r+1)}_d| = |J^{(r)}_d|$, we see that
there exists $r$ such that $J^{(r)}_d$ is $z$-stable. Let $r_d$ be such
that $J^{(r_d)}_t$ is $z$-stable for all $0 \leq t \leq d$. For $d \geq 0$,
let $\mathfrak b^{(d)}$ be the ideal generated by $\bigoplus_{t=0}^d 
J^{(r_d)}_t$. Since $S$ is Noetherian, the ascending chain $\mathfrak
b^{(1)} \subseteq  \mathfrak b^{(2)} \subseteq \cdots$ stabilizes, so 
$J^{(r)} = J^{(r+1)}$ for all $r \gg 0$. Set $J$ to be the stable value.

\end{proof}
 
\begin{lemma}
\label{lemma:stabilizationCharZero}
Let $t > 1$ be an integer, and suppose that $\Bbbk$ contains a primitive
$t$th root of unity $\zeta$.
Assume that $\mathfrak a$ is an $A$-ideal such that $x_i^t \in
\mathfrak a$ for all $1 \leq i \leq n$. Let $I$ be an $B$-ideal such that
$(\mathfrak aB, z^t) \subseteq I$. Then there exists $B$-ideal $J$ such
that $(\mathfrak aB, z^t) \subseteq J$, $H_{I} = H_{J}$ and
$J(B/(\mathfrak aB, z^t))$ is $z$-stable.
\end{lemma}

\begin{proof}
Let $\omega$ be the weight vector on $B$ with $\omega(x_i) = 1$ for all
$i$ and $\omega(z) = 0$. 
Replacing $I$ by $\In_\omega(I)$, 
we may assume that $I$ is multigraded.
Let $L_j$ be the distraction matrix:
\[
\bordermatrix{ & 1 & 2 & \cdots &  t & t+1 & \cdots \cr
x_1 & x_1 & x_1 & \cdots & x_1 &  x_1 & \cdots \cr
\vdots & \vdots & \vdots & \ddots & \vdots & \vdots & \ddots \cr
x_n & x_n & x_n & \cdots & x_n &  x_n & \cdots \cr
z & x_j-z & x_j-\zeta z & \cdots &  x_j-\zeta^{t-1}z & z & \cdots \cr
}
\]
From~\cite{MeMuLPPPurePwrs08}*{Lemma~3.6} we see the following:
\begin{inparaenum}
\item $D_{L_j(z^t)} = x_j^t - z^t$
\item For all $fz^i \in B$, with $f \in A$ and $i \leq t$,
$x_jz^{i-1}$ appears with a nonzero coefficient in $D_{L_j}(fz^i))$.
\end{inparaenum}
Moreover, $(\In_\omega \circ D_{L_j})(\mathfrak aB, z^t) = (\mathfrak aB,
z^t)$. Let $J^{(r)} = (\In_\omega \circ D_{L_j})^r(I)$. Then, arguing as in
the proof of Lemma~\ref{lemma:stabilization}, we see that for all $r \gg
0$, $J^{(r)} = J^{(r+1)}$ and that for all 
$fz^i \in J^{(r)}$ with $z \nmid f$ and $i \leq t$, $fx_jz^{i-1} \in
J^{(r)}$. 
Repeating this argument for all $1 \leq j \leq n$, we complete the proof.
\end{proof}

\section*{Acknowledgements}
We thank A.~Conca and the referee for helpful comments. 
The computer algebra system
\texttt{Macaulay2}~\cite{M2} 
provided valuable assistance in studying examples.


\begin{bibdiv}
\begin{biblist}

\bib{BCRdistr05}{article}{
      author={Bigatti, A.~M.},
      author={Conca, A.},
      author={Robbiano, L.},
       title={Generic initial ideals and distractions},
        date={2005},
        ISSN={0092-7872},
     journal={Comm. Algebra},
      volume={33},
      number={6},
       pages={1709\ndash 1732},
         url={http://dx.doi.org/10.1081/AGB-200058217},
}

\bib{BrHe:CM}{book}{
      author={Bruns, Winfried},
      author={Herzog, J{\"u}rgen},
       title={Cohen-{M}acaulay rings},
      series={Cambridge Studies in Advanced Mathematics},
   publisher={Cambridge University Press},
     address={Cambridge},
        date={1993},
      volume={39},
        ISBN={0-521-41068-1},
}

\bib{BigaUpperBds93}{article}{
      author={Bigatti, Anna~Maria},
       title={Upper bounds for the {B}etti numbers of a given {H}ilbert
  function},
        date={1993},
        ISSN={0092-7872},
     journal={Comm. Algebra},
      volume={21},
      number={7},
       pages={2317\ndash 2334},
         url={http://dx.doi.org/10.1080/00927879308824679},
}

\bib{CHHRigidReslns04}{article}{
      author={Conca, Aldo},
      author={Herzog, J{\"u}rgen},
      author={Hibi, Takayuki},
       title={Rigid resolutions and big {B}etti numbers},
        date={2004},
        ISSN={0010-2571},
     journal={Comment. Math. Helv.},
      volume={79},
      number={4},
       pages={826\ndash 839},
         url={http://dx.doi.org/10.1007/s00014-004-0812-2},
}

\bib{ClemLindMacaulayThm69}{article}{
      author={Clements, G.~F.},
      author={Lindstr{\"o}m, B.},
       title={A generalization of a combinatorial theorem of {M}acaulay},
        date={1969},
     journal={J. Combinatorial Theory},
      volume={7},
       pages={230\ndash 238},
}

\bib{ConcaExtremalGinLex04}{article}{
      author={Conca, Aldo},
       title={Koszul homology and extremal properties of {G}in and {L}ex},
        date={2004},
        ISSN={0002-9947},
     journal={Trans. Amer. Math. Soc.},
      volume={356},
      number={7},
       pages={2945\ndash 2961},
         url={http://dx.doi.org/10.1090/S0002-9947-03-03393-2},
}

\bib{CRVgflags01}{article}{
      author={Conca, Aldo},
      author={Rossi, Maria~Evelina},
      author={Valla, Giuseppe},
       title={Gr\"obner flags and {G}orenstein algebras},
        date={2001},
        ISSN={0010-437X},
     journal={Compositio Math.},
      volume={129},
      number={1},
       pages={95\ndash 121},
}

\bib{EGHcbconj96}{article}{
      author={Eisenbud, David},
      author={Green, Mark},
      author={Harris, Joe},
       title={Cayley-{B}acharach theorems and conjectures},
        date={1996},
        ISSN={0273-0979},
     journal={Bull. Amer. Math. Soc. (N.S.)},
      volume={33},
      number={3},
       pages={295\ndash 324},
         url={http://dx.doi.org/10.1090/S0273-0979-96-00666-0},
}

\bib{eiscommalg}{book}{
      author={Eisenbud, David},
       title={Commutative algebra},
      series={Graduate Texts in Mathematics},
   publisher={Springer-Verlag},
     address={New York},
        date={1995},
      volume={150},
        ISBN={0-387-94268-8; 0-387-94269-6},
        note={With a view toward algebraic geometry},
}

\bib{FrRiLPP07}{incollection}{
      author={Francisco, Christopher~A.},
      author={Richert, Benjamin~P.},
       title={Lex-plus-powers ideals},
        date={2007},
   booktitle={Syzygies and {H}ilbert functions},
      series={Lect. Notes Pure Appl. Math.},
      volume={254},
   publisher={Chapman \& Hall/CRC, Boca Raton, FL},
       pages={113\ndash 144},
}

\bib{GHPtoric08}{article}{
      author={Gasharov, Vesselin},
      author={Horwitz, Noam},
      author={Peeva, Irena},
       title={Hilbert functions over toric rings},
        date={2008},
        ISSN={0026-2285},
     journal={Michigan Math. J.},
      volume={57},
       pages={339\ndash 357},
         url={http://dx.doi.org/10.1307/mmj/1220879413},
        note={Special volume in honor of Melvin Hochster},
}

\bib{GMPveronese10}{article}{
      author={Gasharov, Vesselin},
      author={Murai, Satoshi},
      author={Peeva, Irena},
       title={Hilbert schemes and maximal {B}etti numbers over {V}eronese
  rings},
        date={2011},
        ISSN={0025-5874},
     journal={Math. Z.},
      volume={267},
      number={1-2},
       pages={155\ndash 172},
         url={http://dx.doi.org/10.1007/s00209-009-0614-8},
}

\bib{M2}{misc}{ label={M2},
      author={Grayson, Daniel~R.},
      author={Stillman, Michael~E.},
       title={Macaulay 2, a software system for research in algebraic
  geometry},
        date={2006},
        note={Available at \href{http://www.math.uiuc.edu/Macaulay2/}
  {http://www.math.uiuc.edu/Macaulay2/}},
}

\bib{HuleMaxBettiNos93}{article}{
      author={Hulett, Heather~A.},
       title={Maximum {B}etti numbers of homogeneous ideals with a given
  {H}ilbert function},
        date={1993},
        ISSN={0092-7872},
     journal={Comm. Algebra},
      volume={21},
      number={7},
       pages={2335\ndash 2350},
         url={http://dx.doi.org/10.1080/00927879308824680},
}

\bib{KatonaFiniteSets68}{incollection}{
      author={Katona, G.},
       title={A theorem of finite sets},
        date={1968},
   booktitle={Theory of graphs ({P}roc. {C}olloq., {T}ihany, 1966)},
   publisher={Academic Press},
     address={New York},
       pages={187\ndash 207},
}

\bib{KruskalNoOfSimpl63}{incollection}{
      author={Kruskal, Joseph~B.},
       title={The number of simplices in a complex},
        date={1963},
   booktitle={Mathematical optimization techniques},
   publisher={Univ. of California Press, Berkeley, Calif.},
       pages={251\ndash 278},
}

\bib{MermLexlike06}{article}{
      author={Mermin, Jeffrey},
       title={Lexlike sequences},
        date={2006},
        ISSN={0021-8693},
     journal={J. Algebra},
      volume={303},
      number={1},
       pages={295\ndash 308},
         url={http://dx.doi.org/10.1016/j.jalgebra.2005.11.007},
}

\bib{MerminMonomCI10}{article}{
      author={Mermin, Jeffrey},
       title={Monomial regular sequences},
        date={2010},
        ISSN={0002-9939},
     journal={Proc. Amer. Math. Soc.},
      volume={138},
      number={6},
       pages={1983\ndash 1988 (electronic)},
}

\bib{MeMuColoured10}{article}{
      author={Mermin, Jeff},
      author={Murai, Satoshi},
       title={Betti numbers of lex ideals over some macaulay-lex rings},
        date={2010},
        ISSN={0925-9899},
     journal={J. Algebraic Combin.},
      volume={31},
      number={2},
       pages={299\ndash 318},
         url={http://dx.doi.org/10.1007/s10801-009-0192-1},
}

\bib{MeMuLPPPurePwrs08}{article}{
      author={Mermin, Jeff},
      author={Murai, Satoshi},
       title={The lex-plus-powers conjecture holds for pure powers},
        date={2011},
        ISSN={0001-8708},
     journal={Adv. Math.},
      volume={226},
      number={4},
       pages={3511\ndash 3539},
         url={http://dx.doi.org/10.1016/j.aim.2010.08.022},
}

\bib{MerminPeevaLexifying}{article}{
      author={Mermin, Jeffrey},
      author={Peeva, Irena},
       title={Lexifying ideals},
        date={2006},
        ISSN={1073-2780},
     journal={Math. Res. Lett.},
      volume={13},
      number={2-3},
       pages={409\ndash 422},
}

\bib{MiStCCA05}{book}{
      author={Miller, Ezra},
      author={Sturmfels, Bernd},
       title={Combinatorial commutative algebra},
      series={Graduate Texts in Mathematics},
   publisher={Springer-Verlag},
     address={New York},
        date={2005},
      volume={227},
        ISBN={0-387-22356-8},
}

\bib{PardueDefClass96}{article}{
      author={Pardue, Keith},
       title={Deformation classes of graded modules and maximal {B}etti
  numbers},
        date={1996},
        ISSN={0019-2082},
     journal={Illinois J. Math.},
      volume={40},
      number={4},
       pages={564\ndash 585},
}

\bib{SbarraBdsLocalCoh01}{article}{
      author={Sbarra, Enrico},
       title={Upper bounds for local cohomology for rings with given {H}ilbert
  function},
        date={2001},
        ISSN={0092-7872},
     journal={Comm. Algebra},
      volume={29},
      number={12},
       pages={5383\ndash 5409},
         url={http://dx.doi.org/10.1081/AGB-100107934},
}

\bib{StanEC1}{book}{
      author={Stanley, Richard~P.},
       title={Enumerative combinatorics. {V}ol. 1},
      series={Cambridge Studies in Advanced Mathematics},
   publisher={Cambridge University Press},
     address={Cambridge},
        date={1997},
      volume={49},
        ISBN={0-521-55309-1; 0-521-66351-2},
        note={With a foreword by Gian-Carlo Rota, Corrected reprint of the 1986
  original},
}

\end{biblist}
\end{bibdiv}


\end{document}